\documentclass{amsart}

\usepackage{amsmath,amssymb,amsfonts,amstext,amsthm,bm}
\usepackage{url}
\usepackage{graphicx}
\usepackage{hyperref}
\usepackage{tikz}
\usepackage[all,cmtip]{xy} 
\usepackage{multicol}
\usepackage{verbatim}
\input xy
\xyoption{all}
\usepackage[all]{xy}
\usepackage{xcolor}

\newtheorem{coro}{Corollary}[section]

\newtheorem{dfn}[coro]{Definition}

\newtheorem{lemma}[coro]{Lemma}

\newtheorem{prop}[coro]{Proposition}

\newtheorem{rem}[coro]{Remark}

\newtheorem{thm}[coro]{Theorem}

\newtheorem{ex}[coro]{Example}
\let \fl=\flushleft
\let \ep=\epsilon
\let \ga=\gamma

\let \sub=\subset
\let \be=\beta

\let \al=\alpha
\let \pr=\prime

\let \ov=\overine
\let \sub=\subset
\let \mb=\mathbb
\let \mc=\mathcal
\let \wt=\widetilde

\let \ra=\rightarrow
\let \ov=\overline

\title{Matroids and semirings attached to toric singularity arrangements}
\author{Ethan Cotterill}

\address{IMECC, Universidade Estadual de Campinas (Unicamp), Rua S\'ergio Buarque de Holanda, 651, CEP 13083-859, Campinas, SP, Brazil}

\email{cotterill.ethan@gmail.com}

\author{Cristhian Garay L\'opez}

\address{Centro de Investigaci\'on en Matem\'aticas, A.C. (CIMAT) Jalisco S/N, Col. Valenciana CP. 36023 Guanajuato, Gto, M\'exico}

\email{cristhian.garay@cimat.mx}

\begin{document}
\maketitle

\begin{abstract}
Curve singularities are classical objects of study in algebraic geometry. The key player in their combinatorial structure is the {\it value semigroup}, or its compactification, the {\it value semiring}. One natural problem is to explicitly determine the value semirings of distinguished infinite classes of singularities, with a view to understanding their asymptotic properties. In this paper, we establish a matroidal framework for resolving this problem for singularities determined by arrangements of toric branches; and we obtain precise quantitative results in the case of line arrangements. Our results have implications for the topology of Severi varieties of unisingular rational curves in projective space.
\end{abstract}

\section{Introduction}
The past ten years have been a watershed for the geometry of matroids 
in general, and combinatorial Hodge theory in particular.
While Hodge theory deals primarily with the global geometry of algebraic varieties, in this paper we introduce new combinatorial techniques in the local setting of curve singularities. We show how to associate to any essential arrangement of toric branches a sequence of matroids indexed by degrees $d \in \mb{N}$. For a distinguished subclass of these that includes line arrangements, we show via matroidal techniques that the value semirings of such toric singularity arrangements are generated by elements of minimal support in each degree $d$.

\medskip
More concretely now, fix positive integers $r\geq n \geq 2$; let $F$ denote a field of cardinality $\# F \geq r$; and fix an $n$-tuple of positive integers $m=(m_1,\dots,m_n) \in \mb{N}_{>0}^n$. Let $X\subset \mathbb{A}^n_{F}$ denote an essential arrangement of $r$ distinct toric branches of the form $(v_{1,j} t_1^{m_1},\dots, v_{n,j} t_n^{m_n})$, where $v_{i,j} \in F$ for every $i=1,\dots,n$ and $j=1,\dots,r$.
When the arrangement is linearly nondegenerate,
it determines a singularity with $r$ branches and embedding dimension $n$.

\medskip
The value semiring ${\rm S}(X)$ of the singularity ${\bf 0} \in X$ is a 
sub-semiring of $\overline{\mathbb{N}}^r$, where $\overline{\mb{N}}= (\mb{N} \cup \{\infty\},\text{min},+)$. By definition, there exists a surjection $\pi_X: F[x_1,\ldots,x_n]\ra {\rm S}(X)$. For every $d \in \mb{N}$, let $R(d):= \bigoplus_{\{I\::\:\langle I,m \rangle=d\}} F x^I$; we have $F[x_1,\ldots,x_n]= \bigoplus_{d \geq 0} R(d)$.
For each $d \in \mb{N}_{>0}$ we distinguish {\it $d$-forms of minimal support}: these are elements of $R(d)$ whose $\pi_X$-images are of minimal support with respect to set-theoretic inclusion. Our main structural result for value semirings of toric arrangements is the following.

\begin{thm}\label{semiring_generation_thm}(Theorem~\ref{semiring_generation_theorem})
Whenever $\#F \geq r$, the
$\pi_X$-images of $d$-forms, $d>0$ of minimal nonempty support of a toric arrangement $X\subset \mathbb{A}^n_F$ of distinguished type as above generate the semiring ${\rm S}(X)$. 
\end{thm}

Theorem~\ref{semiring_generation_thm} opens the door to the matroidal study of these curve singularities, and in particular of multiple points. Nontrivial connected-sum decompositions of matroids generalize realizations of $r$-fold multiple points in $F^n$ as unions of $r_i$-fold multiple points in $F^{n_i}$, where $r_i<r$ and $n_i<n$. A {\it uniform} matroid generalizes a multiple point associated with a configuration for which every subconfiguration has embedding dimension equal to its number of branches. 

\medskip
Our results bear some similarities 
to those of \cite{MR}, where the authors show that the 
supports in $\mathbb{B}[x_1,\ldots,x_n]$ of those polynomials belonging to an ideal $I\subset F[x_1,\ldots,x_n]$ 
prescribe and are uniquely prescribed by the family of matroids $\{M(I^h_d):d>0\}$ of vector spaces $I^h_d\subset F[x_0,x_1,\ldots,x_n]$, where $I^h_d$ denotes the degree $d$ part of the homogenization $I^h$ of  $I$. More precisely, according to \cite[Remark 2.3]{MR2}, each $M(I^h_d)$ is the $\mathbb{B}$-semimodule generated by the polynomials Supp$(g)$ with $g \in I^h_d$; however, there may be some elements in $M(I^h_d)$ that are not realizable as the supports of elements in $I^h_d$. On this point our work differs, since by definition the value semiring ${\rm S}(X)$ consists precisely of those supports that are realized 
by elements in the local algebra of $X$. Accordingly we are forced to work with {\it scrawls} rather than with circuits.

\medskip
In the last section, we focus on the case of essential {\it line} arrangments, whose value semirings ${\rm S}(X)$ are exactly those of $r$-fold multiple points. We introduce {\it very uniform} essential line  arrangements, which generalize essential line arrangements whose underlying matroids are uniform. These satisfy a maximal-rank property for generalized Vandermonde matrices related to Veronese embeddings of $\mb{P}^{n-1}$ that is spelled out in Lemma~\ref{interpolation_lemma}.
Whenever this maximal-rank property is satisfied, we obtain the following explicit, if algorithmic, description of ${\rm S}(X)$.

\begin{thm}(Theorem~\ref{value_semiring_very_uniform_case})
Given positive integers $n$ and $r$, let $X$ denote a very uniform essential line arrangement with $r$ branches and embedding dimension $n$. Let $d$ be the minimal positive integer for which $\binom{d+n-1}{n-1} \geq r$; the conductor ${\rm S}(X)$ is then $c=(d,\dots,d)$. On the other hand, for every positive integer $d$ for which $\binom{d+n-1}{n-1} < r$, arbitrary permutations of the vector $(d,\dots,d,d-1,\dots,d-1)$ featuring $i_d \in \{d,\dots,r-1\}$ instances of $d$, is a gap of ${\rm S}(X)$; and conversely, every gap of ${\rm S}(X)$ is of this form.
\end{thm}

While Theorem~\ref{value_semiring_very_uniform_case} gives a complete characterization of ${\rm S}(X)$ when $X$ is a very uniform line arrangement, the combinatorics of ${\rm S}(X)$ remains somewhat mysterious whenever $n>2$. On the other hand, whenever $n=2$, it is not hard to check that the {\it genus} of $X$ (i.e., the delta-invariant of the associated singularity) is $g(X)=\binom{r}{2}$; see Example~\ref{planar_case}. We use this fact, coupled with an upper bound on the number of algebraically independent conditions imposed on rational curves by $r$-fold multiple points of embedding dimension $\ell$, to produce novel examples of {\it Severi varieties} of unisingular rational curves.

\subsection{Set-theoretic conventions} Given a positive integer $n$ and $I=(i_1,\ldots,i_n)\in \mathbb{N}^n$, we let $|I|:=i_1+\cdots+i_n$, and 
if $E$ is a finite set, we let $\binom{E}{n}$ denote the collection of cardinality-$n$ subsets of $E$. Similarly (and consistently), we let $2^E:=\{I\subset E\}$ denote the power set of $E$. We let $[r]$ denote $\{1,\ldots,r\}$. 

\section{Semirings and matroids}
\subsection{Basic notions for idempotent semirings}
Let $\mathbb{B}=(\{0,\infty\},\text{min},+)$ denote the Boolean semifield. Sometimes we write $\text{min}\{a,b\}=a\oplus b$ and $a+b=a\odot b$. 
Note that the category of idempotent semirings is precisely the category of $\mathbb{B}$-algebras, and the category of idempotent monoids 
is the category of $\mathbb{B}$-semimodules. Moreover, every idempotent semiring ${\rm S}=({\rm S},+,\times)$ is {\it sharp}, i.e., there are no nontrivial equalities $\sum_ia_i=0$ with $a_i \in {\rm S}$.

\medskip
Now fix a positive integer $r\geq1$. Note that the power set $2^{[r]}$ is a semiring with respect to $+:=\cup$ and $\times:=\cap$. The {\it support} map $\text{Supp}: {\mathbb{B}}^r \ra 2^{[r]}$ that sends $a$ to $\{i\in[r]:a_i\neq\infty\}$ is an isomorphism of semirings, so we may (and shall) identify a subset $A$ of $[r]$ with its indicator vector $a$ in $ {\mathbb{B}}^r$. Let $a^{\pr}$ denote the {\it complement} of $a\in {\mathbb{B}}^r$, i.e., the unique $b$ such that $a\oplus b=\mathbf{0}$ and $a\odot b=\boldsymbol{\infty}=(\infty,\dots,\infty)$. According to de Morgan's laws, we have $(a\oplus b)^{\pr}=a^{\pr}\odot b^{\pr}$ and $(a\odot b)^{\pr}=a^{\pr}\oplus b^{\pr}$. 

\medskip
Let $\overline{\mathbb{N}}^r=(\ov{\mathbb{N}}^r,\odot,\mathbf{0},\oplus,\boldsymbol{\infty})$ denote the idempotent semiring with operations
\[
a\odot b=(a_1+b_1,\ldots,a_r+b_r) \text{ and } a\oplus b=(\text{min} \{a_1,b_1\},\ldots,\text{min} \{a_r,b_r\})
\]
whose identities with respect to $\odot$ and $\oplus$, respectively, are
$\mathbf{0}$ and $\boldsymbol{\infty}$. Given $m\in \overline{\mathbb{N}}$ and $a\in \overline{\mathbb{N}}^r$, we write $m\cdot a=ma=(a_1+m,\ldots,a_r+m)$, and we let $\mathbf{m}=(m,\ldots,m)\in \overline{\mathbb{N}}^r$. 
The idempotent order on $\overline{\mathbb{N}}^r$ is defined by $a\leq b$ if and only if $a\oplus b=a$; 
equivalently, this means that
\begin{equation}
\label{eq_product_order}
    a\leq b \Leftrightarrow a_i\leq b_i \text{ for every } i=1,\ldots, r.
\end{equation}
In other words, the idempotent order on $\overline{\mathbb{N}}^r$ is the product of $r$ copies of the usual order on $\ov{\mathbb{N}}$. Note that $\mathbf{0}\leq a \leq \boldsymbol{\infty}$ for all $a\in \overline{\mathbb{N}}^r$.

\medskip
The support map extends in an obvious way to give ``support" maps $\overline{\mathbb{N}}^r \ra 2^{[r]}$ and $\overline{\mathbb{N}}^r \ra  {\mathbb{B}}^r$. Indeed, viewed as an element of $\mb{B}^r$, the support $\text{Supp}(a)$ of $a\in \overline{\mathbb{N}}^r$ is the vector in $\mathbb{B}^r$ whose $i$-th coordinate equal to $0$ if $a_i\neq\infty$ and $\infty$ otherwise. When we want to distinguish the support valued in $2^{[r]}$, we will follow \cite{CH} and write $I_a$ in place of $\text{Supp}(a)$.

\medskip
We say that $a\in \overline{\mathbb{N}}^r$ is {\it homogeneous} if $a=d\cdot\text{Supp}(a)$ for some $d\neq\infty$, which we also denote by $\text{deg}(a)$. 
As products (with respect to $\odot$) of homogeneous elements are clearly homogeneous, any $a\in \overline{\mathbb{N}}^r$ may be expressed uniquely as a sum (with respect to $\oplus$) of homogeneous elements; and $\overline{\mathbb{N}}^r$ decomposes as a direct sum of homogeneous pieces.

\begin{rem}
\label{rem_subsemring}
The Boolean semifield $\mathbb{B}^r$ embeds naturally as a sub-semiring of $\overline{\mb{N}}^r$, namely as its degree-zero piece. 
\end{rem}

A subsemiring ${\rm S} \sub \overline{\mb{N}}^r$ is {\it finitely generated} if there exist finitely many
$\phi_1,\dots,\phi_m \in \overline{\mb{N}}^r$ in terms of which every element of ${\rm S}$ may be represented as $\phi^{\odot I_1}\oplus \cdots \oplus \phi^{\odot I_k}$,
where $\phi^{\odot I}=\phi_1^{\odot i_1}\odot\cdots\odot\phi_m^{\odot i_m}$ whenever $I=(i_1,\dots,i_m)$. 
This means that ${\rm S}$ is the image of the semiring morphism  $\mathbb{B}[x_1,\ldots,x_m] \ra \overline{\mathbb{N}}^r$ sending $x_i$ to $\phi_i$; so as a shorthand we write ${\rm S}=\mathbb{B}[\phi_1,\ldots,\phi_m]$.

\medskip
The {\it conductor} of a subsemiring ${\rm S} \sub \overline{\mb{N}}^r$ is the smallest $r$-tuple $c=c({\rm S})\in\overline{\mb{N}}^r$ 
with the property that every $a\in \overline{\mb{N}}^r$ with $a \geq c$ belongs to ${\rm S}$, provided such an element exists. 

\subsection{Basic notions for matroids}\label{Section_Matroids}
We follow \cite{JO} and \cite[Ch. 2.2]{Welsh}. A matroid is a combinatorial object that satisfies several cryptomorphic sets of axioms. Two of these formulations will be relevant in the sequel.
Hereafter we fix positive integers $r\geq n\geq2$.

\medskip
The most familiar way of specifying a matroid is in terms of a {\it rank function} $\rho:2^{[r]}\ra \mathbb{N}$ for which $\rho(\emptyset)=0$ and that moreover satisfies:
\begin{enumerate}
    \item $\rho(A\cup B)+\rho(A\cap B)\leq \rho(A)+\rho(B)$ for all $A,B\subset [r]$; and 
    \item $\rho(A)\leq \rho(A\cup\{i\})\leq \rho(A)+1$ for every $A\subset [r]$ and $i\in[r]\setminus A$.
\end{enumerate}
The pair $M=([r],\rho)$ is then a matroid of {\it rank} equal to $\rho([r])$. An {\it independent set} of $M$ is a subset of $[r]$ whose rank equals its cardinality; any non-independent subset of $[r]$ is {\it dependent}. A {\it basis} is a maximal independent set and a {\it circuit} of $M$ is a minimal dependent subset. An index 
$i \in [r]$ is a {\it loop} whenever it lies outside every basis of $M$. Likewise, a subset of $[r]$ is a {\it flat} of $M$ whenever it is maximal for its rank. A basic fact is that the flats of $M$ form a lattice, which we denote by $L(M)$.

\begin{dfn}
   A matroid $M=([r],\rho)$ is {\it simple} provided $\rho(i)=1$ for every $i\in [r]$, and $\rho(i,j)=2$ for every $i\neq j\in [r]$. 
\end{dfn}

Every simple matroid is determined by its lattice of flats. Moreover, every matroid $M=([r],\rho)$ possesses a canonical {\it simplification} $\widetilde{M}=([s],\rho’)$,
where $[s]=([r]-L(M))/\sim$, $L(M)$ is the set of loops of $M$, and $\sim$ identifies parallel elements; see \cite{MB}. The characteristic property of $\wt{M}$ is that the induced map $\sigma:2^{[r]}\xrightarrow[]{}2^{[s]}$ makes the following diagram commutative:
\begin{equation}    \xymatrix{2^{[r]}\ar[r]^{\rho}\ar[dr]_{\sigma}&\mathbb{N}\\
    &2^{[s]}\ar[u]_{{\rho'}}}
\end{equation}

An alternative and relatively new way of specifying matroids is via sets of {\it scrawls}, as in \cite{ABFG23}. These are more convenient when working with matroids induced by families of supports of vector spaces, as we will see later.

\begin{dfn}
\label{dfn_sea}
    A subset $\mathcal{S}$ of $ 2^{[r]}$ is a {set of scrawls} if it satisfies the following axioms:
\begin{enumerate}
\item Every scrawl is non-empty;
    \item $\mathcal{S}$ is a semigroup with respect to unions, i.e. any union of elements in $\mathcal{S}$ belongs to $\mathcal{S}$; and
    \item $\mathcal{S}$ satisfies the {scrawl} elimination axiom: Given distinct scrawls $S_1,S_2$ and $e \in S_1 \cap S_2$, then $(S_1 \cup S_2) \setminus \{ e \}$ contains a scrawl.
\end{enumerate}
\end{dfn}

Circuits and scrawls are closely related. Indeed, the minimal elements of a set of scrawls satisfy the axioms for circuits, while the set of unions of circuits satisfy the axioms for scrawls. \footnote{In \cite{MR2}, the term {\it cycle} is used in place of scrawl.}

\begin{rem}
Circuits and scrawls may be characterized algebraically. Indeed, via the semiring isomorphism $2^{[r]}\xrightarrow[]{}\mathbb{B}^r$, a $\mathbb{B}$-semimodule $M=(M,\boldsymbol{\infty},\oplus)$ is matroidal if and only if $M\setminus\{\boldsymbol{\infty}\}$ satisfies the scrawl elimination axiom from Definition \ref{dfn_sea}. Consequently, sets of scrawls are equivalent to matroidal  $\mathbb{B}$-semimodules, and elements $a\in M\setminus\{\boldsymbol{\infty}\}$ of minimal support  are additively irreducible, in the sense that for every factorization $a=b\oplus c$, we have $a=b$ or $b=c$. See Definition ~\ref{maximal_irreducible} for the corresponding notion of multiplicative irreducibility.
\end{rem}

\medskip
In what follows, we will make critical use of the following result \cite[Theorem 3.5]{ABFG23}: For every $F$-vector space $W\subset F^r$ with $\#F\geq r$, $\text{Supp}(W)\subset \mathbb{B}^r$ is a matroidal  $\mathbb{B}$-semimodule. It follows, in particular, that the set $C$ of elements of minimal support in $\text{Supp}(W\setminus\{0\})$ is the set of circuits of a matroid, and $\text{Supp}(W)$ is the matroidal  $\mathbb{B}$-semimodule spanned by (unions of elements in) $C$.

\section{Singularities with toric branches}\label{sings_with_toric_branches}
It is well-known that over $\mb{C}$ every cusp degenerates in a flat family to a toric branch, i.e., to a cusp parameterized by monomials in a uniformizing parameter; see \cite[p.49]{T} for an explicit construction that works over an arbitrary algebraically closed field. 
Degenerating branch by branch, it follows that {\it any} curve singularity degenerates in a flat family to a union of toric branches, and singularities with toric branches comprise a distinguished class.\footnote{It seems reasonable to further speculate that {\it every} value semigroup of a singularity arises from a singularity with toric branches.} 

\medskip
Now suppose $F$ is a field and that $X \sub \mb{A}^n_{F}$ is a singularity of embedding dimension $n\geq2$, with $r\geq1$ toric branches. For $i=1,\dots,r$, let $t_i\mapsto (v_{i,1}t_i^{m_{i,1}},\ldots,v_{i,n}t_i^{m_{i,n}})$ denote a parameterization of the $i$-th branch, where $v_{i,j}\in F$ and $m_{i,j}\in\mathbb{N}$. 
Let $A=(v_{i,j}t_i^{m_{i,j}})_{i,j}$  denote the associated $r \times n$ matrix, 
and for every $i=1,\ldots,r$, let $v_i=(v_{i,1},\ldots,v_{i,n})$ and $m_i=(m_{i,1},\ldots,m_{i,n})$. By convention, we set $m_{i,j}:=0$ whenever $v_{i,j}=0$.


\begin{dfn} Let $\mu_A:F[x_1,\ldots,x_n]\xrightarrow{} \prod_{i=1}^r F[\![t_i]\!]$ denote the map obtained by $F$-linearly extending the assignment $x_j\mapsto(v_{1,j}t_1^{m_{1,j}},\ldots,v_{r,j}t_r^{m_{r,j}})$ for $j=1,\ldots,n$.
\end{dfn}

This is a homomorphism of rings since it is an evaluation map. We now describe the map $\mu_A$. Given $f=\sum_Ia_Ix^I\in F[x_1,\ldots,x_n]$ and $i=1,\ldots,r$, let $\sum_{s\geq0 }f_s^{(i)}(x)$ denote the graded decomposition of $f$ with respect to the weight vector $(|x_1|_i,\ldots,|x_n|_i)=(m_{i,1},\ldots,m_{i,n})$. 


\begin{prop}\label{muAf} Let $A=(v_{i,j}t_i^{m_{i,j}})_{i,j}$ be as above and let $f=\sum_Ia_Ix^I\in F[x_1,\ldots,x_n]$. We have

\begin{equation}\label{eq_muAf}
\mu_A(f)=\biggl(\sum_sf_s^{(1)}(v_1)t_1^s,\ldots,\sum_sf_s^{(r)}(v_r)t_r^s\biggr).
\end{equation}
\end{prop}

\begin{proof} Given $I=(i_1,\ldots,i_n)\in\mathbb{N}^n$, we have 
\[
\mu_A(x^I)=\bigl({v_1}^It_1^{\left<I,m_1\right>},\ldots,{v_r}^It_r^{\left<I,m_r\right>}\bigr)
\]
where $v_i=(v_{i,1},\ldots,v_{i,n})$ and $m_i=(m_{i,1},\ldots,m_{i,n})$. More generally, for every $f=\sum_Ia_Ix^I\in F[x_1,\ldots,x_n]$, the fact that $\mu_A$ is $F$-linear in monomial arguments implies that
\[
\begin{split}
\mu_A(f)&=
\sum_Ia_I\bigl({v_1}^It_1^{\left<I,m_1\right>},\ldots,{v_r}^It_r^{\left<I,m_r\right>}\bigr)\\
&=\biggl(\sum_s(\sum_{\left<I,m_1\right>=s}a_I{v_1}^I)t_1^{s},\ldots,\sum_s(\sum_{\left<I,m_r\right>=s}a_I{v_r}^I)t_r^{s}\biggr)\\
&=\biggl(\sum_sf_s^{(1)}(v_1)t_1^s,\ldots,\sum_sf_s^{(r)}(v_r)t_r^s\biggr).\qedhere
\end{split}
\]
\end{proof}

Now let $\text{ord}:\prod_{i=1}^r F[\![t_i]\!]\xrightarrow{}\overline{\mathbb{N}}^r$ denote the map that sends $(g_1(t_1),\ldots,g_r(t_r))$ to $(\text{ord}_{t_1}(g_1(t_1)),\ldots,\text{ord}_{t_r}(g_r(t_r)))$, where $\text{ord}_t:F[\![t]\!] \xrightarrow{}\overline{\mathbb{N}}$ is the $t$-adic valuation. Let $\text{ord}_A:F[x_1,\ldots,x_n]\xrightarrow{}\overline{\mathbb{N}}^r$ denote the composition $\text{ord}\circ\mu_A$. 

We recall that $\overline{\mathbb{N}}^r$ is the idempotent semiring endowed with coordinate-wise tropical addition $\oplus$ and multiplication $\odot$, and with the order $\leq$ characterized by
$a\leq b $ if and only if $a\oplus b=a$; this in turn yields the product order derived from the usual order on $\overline{\mathbb{N}}$; that is,
$a\leq b$ if and only if $a_i\leq b_i$ for every $i=1,\ldots, r$.

\begin{prop}\label{Prop_Valuation}
The map $\text{ord}_A:F[x_1,\ldots,x_n]\xrightarrow{} \overline{\mathbb{N}}^r$ is a non-archimedean valuation.
\end{prop}

\begin{proof}
Following \cite{GG}, it suffices to show that
\begin{enumerate}
    \item $\text{ord}_A(0)=\boldsymbol{\infty}$;
    \item $\text{ord}_A(1)=\text{ord}_A(-1)=\mathbf{0}$;
    \item $\text{ord}_A(f+g)\geq \text{ord}_A(f)\oplus \text{ord}_A(g)$; and
    \item $\text{ord}_A(fg)=\text{ord}_A(f)\odot \text{ord}_A(g)$.
\end{enumerate}
Each of these properties follows immediately from the fact that $ord_A=ord\circ \mu_A$.
\end{proof}

It follows from Proposition~\ref{Prop_Valuation} that 
$\text{ord}_A(f+g)_i=\text{ord}_A(f)_i\oplus \text{ord}_A(g)_i$ whenever $\text{ord}_A(f)_i\neq \text{ord}_A(g)_i$, and 
$\text{ord}_A(f+g)=\text{ord}_A(f)\oplus \text{ord}_A(g)$ whenever $\text{ord}_A(f)_i\neq \text{ord}_A(g)_i$ for every $i=1,\ldots,r$.
The following result characterizes $\text{ord}_A$ explicitly.

\begin{coro} Let $A=(v_{i,j}t_i^{m_{i,j}})_{i,j}$ be as above. Given $f=\sum_Ia_Ix^I\in F[x_1,\ldots,x_n]$, let $\mu_A(f)$ be as in equation \eqref{eq_muAf}. We have
\begin{equation}
\label{p_explicit} 
    \text{ord}_A(f)_i=ord_{t_i}\biggl(\sum_sf_s^{(i)}(v_i)t_i^s\biggr)=\begin{cases}
        \text{min}\{s:f_s^{(i)}(v_i)\neq0\},&\text{ if such a minimal }s\text{ exists};\\
    \infty &\text{otherwise}.\\
    \end{cases}
\end{equation}
\end{coro}

\begin{proof} The result follows from Propositions~\ref{muAf} and \ref{Prop_Valuation}.
\end{proof}




\medskip
Suppose now that exponent vectors of the matrix $A=(v_{i,j}t_i^{m_{i,j}})_{i,j}$ are constant, i.e., that $m_i=m$ for every $i=1,\ldots,r$. We have 
$F[x_1,\ldots,x_n]= \bigoplus_{d \geq 0} R(d)$, where $R(d):= \bigoplus_{I: \langle I,m \rangle=d} F x^I$ for every $d \in \mb{N}$. Recall from the introduction that for each $d \in \mb{N}_{>0}$, a {\it $d$-form of minimal support} is any element of $R(d)$ whose $\text{ord}_A$-image is of minimal support. Let $ev_{(v_1,\ldots, v_r)}$ denote the linear evaluation map $R(d)\to F^r$ given by $f\mapsto (f(v_1),\ldots,f(v_r))$; and let ${\rm S}(X,d):=\text{ord}_A(R(d))$.

\begin{prop}
\label{prop_ph}
Whenever $\#F\geq r$, we have that ${\rm S}(X,d)=\text{ord}_A(R(d))\subset\overline{\mathbb{N}}^r$ is a matroidal  $\mathbb{B}$-semimodule spanned by 
\[
\left\{ d\cdot\text{Supp}(ev_{(v_1,\ldots, v_r)}(f)):\text{Supp}(ev_{(v_1,\ldots, v_r)}(f))\in\mathbb{B}^r \setminus \{\boldsymbol{\infty}\} \text{ is minimal}\right\}.
\]
\end{prop}
\begin{proof}
    Assume that $f\in R(d)$. According to equation~\eqref{p_explicit}, we have 
    \[
        \text{ord}_A(f)_i=\begin{cases}
            d,&\text{ if }f(v_i)\neq0,\\
            \infty,&\text{ otherwise.}\\
        \end{cases}
    \]
Thus $\text{ord}_A(f)=d\cdot\text{Supp}(ev_{(v_1,\ldots, v_r)}(f))$. On the other hand, whenever $\#F\geq r$, \cite[Theorem 3.5]{ABFG23} establishes that $\text{Supp}(ev_{(v_1,\ldots, v_r)}(R(d))$ coincides with the matroidal  $\mathbb{B}$-semimodule generated by the set of elements of minimal support.
\end{proof}

\begin{thm}\label{semiring_generation_theorem}
Whenever $\#F\geq r$, the $\pi_X$-images of $d$-forms, $d>0$ of minimal nonempty support of a toric arrangement $X\subset \mathbb{A}^n_F$ of distinguished type as above generate the semiring ${\rm S}(X)$.
\end{thm}

\begin{proof}
    Given $f\in F[x_1,\ldots,x_n]$, the fact that $m_i=m$ implies that $f_d^{(i)}=f_d$ for every $i=1,\ldots,r$. It follows that
    \[
    \begin{split}
    \mu_A(f)=(\sum_df_d(v_1)t_1^d,\ldots,\sum_df_d(v_r)t_r^d)=\sum_d\mu_A(f_d)
    \text{ and}\\
\text{ord}_A(f)=\bigoplus_d\text{ord}_A(\mu_A(f_d))=\bigoplus_d\text{ord}_A(f_d)=\bigoplus_i d\cdot\text{Supp}(ev_{(v_1,\ldots, v_r)}(f_d)).
\end{split}
\]
On the other hand, by Proposition \ref{prop_ph}, $\text{Supp}(ev_{(v_1,\ldots, v_r)}(f_d))$ is generated by those elements of minimal support of $\text{Supp}(ev_{(v_1,\ldots, v_r)}(R(d))$ for every integer $d \geq 0$. 
\end{proof}

Summing all the ${\rm S}(X,d)\subset \overline{\mathbb{N}}^r$ yields ${\rm S}(X)=\sum_d{\rm S}(X,d)$, where each ${\rm S}(X,d)$ is a matroidal  $\mathbb{B}$-semimodule, or a tropical linear space. This effectively concludes our discussion of the additive structure of the semiring ${\rm S}(X)$. 
In the setting of tropical schemes \cite{MR}, the authors produce a similar decomposition  $\text{trop}(I^h)=\sum_dM(I_h^d)$ into $\mathbb{B}$-semimodules $M(I_h^d)\subset \mathbb{B}[x_1,\ldots,x_n]$ of cycles (i.e., scrawls) for the tropicalization (with respect to the trivial valuation on $F$) of the homogeneization $I^h$ of an ideal $I\subset F[x_1,\ldots,x_n]$.

\medskip
However, in our case, the multiplicative structure of the semiring ${\rm S}(X)$ also plays an important r\^ole. 
Our next task is to show that \emph{absolute maxima} $a \in {\rm S}(X)$ in the sense of \cite[Sec. 3]{CH} may be reformulated in terms of the semiring order of ${\rm S}(X)$.



\begin{dfn}\label{maximal_irreducible} An element $a\in {\rm S}(X)$ is
\begin{enumerate}
    \item {\it multiplicatively irreducible} if for every factorization $a=b\odot c$, we have $a=b$ or $b=c$;
    \item {\it maximal} if whenever $a\leq b$, we have $b=\boldsymbol{\infty}$ or $b=a$;
    \item a {\it generator}, if it is both irreducible and maximal.
\end{enumerate}
\end{dfn}

\begin{lemma}\label{maximal_equivalence} An element $a \in {\rm S}(X)$ is maximal if and only if $a$ is an \emph{absolute maximum} in the sense of \cite[Sec. 3]{CH}.
\end{lemma}
\begin{proof}
Following \cite{CH}, given any $a \in {\rm S}$ and any proper subset $J$ of $\text{Supp}(a)$, we set
\[
F_J(s):=\{b \in {\rm S}: b_i>a_i \text{ for } i \in \text{Supp}(a)\setminus J \text{ and } a_j=b_j \text{ for } j \notin \text{Supp}(a)\setminus J.\]

    Suppose that $a \in {\rm S}$ is not maximal in the sense of Definition~\ref{maximal_irreducible}. Then there is some $b\notin \{a,\boldsymbol{\infty}\}$ for which $a \leq b$, and it follows that $\text{Supp}(b)\subseteq \text{Supp}(a)$.
    Furthermore, 
    the fact that $a\oplus b=a$ implies there is some $i\in \text{Supp}(b)$ for which $a_i<b_i$, which in turn means there is a proper subset $J\subset \text{Supp}(a)$ for which $b\in F_J(a)$; and thus $a$ is not an absolute maximum of ${\rm S}$ in the sense of \cite[Sec. 3]{CH}.

\medskip
    Conversely, suppose $a$ is not an absolute maximum of ${\rm S}$. Then $a\neq\boldsymbol{\infty}$ and there is some proper subset $J\subset \text{Supp}(a)$ for which $b\in F_J(a)$. Clearly $b\notin \{a,\boldsymbol{\infty}\}$ and $a\oplus b=a$, which means precisely that $a\leq b$.
\end{proof}

\begin{coro}\label{minimal_support_maxima}
Assume that $F=\ov{F}$ is algebraically closed, and that $X$ is a toric arrangement of distinguished type. For every $d>0$, the $\pi_X$-images of $d$-forms of minimal support are maxima of ${\rm S}(X)$.
\end{coro}

\begin{proof}
This is an immediate consequence of \cite[Thm. 19]{CH}.
\end{proof}

\section{Applications to line arrangements}

Hereafter, $F$ denotes an algebraically closed field of characteristic that is either zero or sufficiently large, and $X\subset \mathbb{A}^n_F$ denotes an essential line arrangement 
of distinguished type induced by a matrix $A=(v_{i,j}t_i^{m_{i,j}})_{i,j}$ with $m_{i,j}=1$. In this case, $X$ determines a matroid $M_X$ of rank $n$ on the set of labels $[r]$ by choosing $\{v_i\in X_i\setminus\{0\}\::\:i=1,\ldots,r\}$, which is always simple.

\subsection{The value semiring of an essential line arrangement}
\label{Subsection_line_arr}
Let $\{e_1,\ldots,e_r\}$ denote the standard vector basis of $F^r$. In this subsection, we use the results of the preceding section to describe ${\rm S}(X)$ in some interesting cases.

\begin{ex}
\label{axes_conf}
Suppose $r=n$. In this case, $M_X$ is the uniform matroid $U_r^r$ of rank $r$ on $[r]$.
Let $A$ be the $r \times r$ matrix whose $j$-th line corresponds to a parameterization of the $j$-th branch of $X$. The isomorphism class of $X$ is clearly invariant under elementary column operations; so without loss of generality, we may assume $A$ is the $r \times r$ identity matrix. The valuative image of the $j$-th line $e_j$ of $A$ is $\ep_j \in \ov{\mb{N}}^r$, whose $j$-th entry is 1 and whose remaining entries are $\infty$. It follows that the conductor $c$ of ${\rm S}(X)$ is the all-ones vector $\mathbf{1}=(1,\dots,1)$, and that ${\rm S}(X)$ is (additively) generated by $\ep_j$, $j=1,\dots,r$. 

\emph{Indeed, for every $a=(a_1,\dots,a_r) \in \ov{\mb{N}}^r$ with $1 \leq a_i< \infty$ for every $i$, we have $a=\bigoplus_{i=1}^r a_i \ep_i$, which shows that $c \leq \mathbf{1}$ and that ${\rm S}(X)$ is generated by $\ep_j$, $j=1,\dots,r$. On the other hand, no nontrivial vector $v \in \ov{\mb{N}}^r$ strictly less than $\mathbf{1}$ may be realized as a sum of vectors $\ep_j$ in $\ov{\mb{N}}^r$; so every such $v$ belongs to $\ov{\mb{N}}^r \setminus {\rm S}(X)$. In particular, the $r-1$ vectors $(1,0,\dots,0)$, $(1,1,0,\dots,0)$, \dots, $(1,1,\dots,1,0)$ belong to $\ov{\mb{N}}^r \setminus {\rm S}(X)$ and determine the gap sequence of a \emph{saturated path} between $\mathbf{0}=(0,\dots,0)$ and $c=\mathbf{1}$; so $X$ has genus $g(X)=r-1$.}

\end{ex}

\begin{ex}\label{planar_case}
Suppose $n=2$. In this case, $M_X$ is a simple matroid of rank $2$ on $[r]$, it is thus the uniform matroid $U_r^2$. The value semiring ${\rm S}(X)$ has conductor $c=\mathbf{r-1}=(r-1,\dots,r-1)$ and genus $g(X)=\binom{r}{2}$.

\emph{Indeed, according to Corollary~\ref{minimal_support_maxima}, the semiring ${\rm S}(X)$ is generated by homogeneous degree-$d$ vectors of minimal support, i.e., vectors that are permutations of $(d,\dots,d,\infty,\dots,\infty)$ for some fixed degree $d \in \mb{N}_{>0}$, in which the number of instances of $d$ is minimal. Our first task is to explicitly identify these.}

\emph{To this end, let $t_i$ denote the local coordinate for the $i$-th branch of a (parameterization of) $X$. Applying elementary column operations if necessary, we may assume without loss of generality that the first two rows of the associated matrix $A$ are $(t_1,0)$ and $(0,t_2)$; and that every subsequent row is of the form $(\al_i t_i, \be_i t_i)$, where $\al_i, \be_i \in F^*$.}

\emph{By construction, homogeneous vectors of minimal support and degree $d$ are precisely the valuative images of degree-$d$ vector-valued polynomials in $\mu_A(x_1)=(\al_i t_i)_{i=1}^r$ and $\mu_A(x_2)=(\be_i t_i)_{i=1}^r$, where $\al_1=1$, $\be_1=0$, $\al_2=0$, $\be_2=1$, and $\al_i,\be_i \in F^*$ for every $i \geq 3$. Any such polynomial is of the form}
\[
P_d=\sum_{j=0}^d \ga_j ((\al_i t_i)_{i=1}^r)^j ((\be_i t_i)_{i=1}^r)^{d-j}= \bigg(\sum_{j=0}^d \ga_j \al_i^j \be_i^{d-j} t_i^d \bigg)_{i=1}^r.
\]
\emph{Now say $d \leq r-1$. Because the points $(\al_i, \be_i) \in F^2$, $i=1,\dots,r$ have distinct slopes, the $r \times (d+1)$ coefficient matrix $(\al_i^j \be_i^{d-j})_{i,j}$ has rank $d+1$; indeed, this follows from the fact that Vandermonde matrices are of maximal rank. Accordingly, at most $d$ of the equations $\sum_{j=0}^d \ga_j \al_i^j \be_i^{d-j}=0$ may be simultaneously satisfied without $P_d$ being the all-zeroes vector. Moreover, \emph{any} $r-d$ of these may be satisfied. It follows immediately that those homogeneous degree-$d$ vectors of minimal support in ${\rm S}(X)$ comprise all permutations of $(d,\dots,d,\infty,\dots,\infty)$ in which $d$ appears at least $r-d$ times.}

\emph{One immediate consequence is that the conductor of ${\rm S}(X)$ is $c=\mathbf{r-1}=(r-1,\dots,r-1)$. Another is that the $\binom{r}{2}$ vectors of the form $(d,\dots,d,d-1,\dots,d-1)$ featuring $i_d \in \{d,\dots,r-1\}$ instances of $d$, for every $d \in \{1,\dots,r-2\}$, determine the gap sequence of a saturated path between $(0,\dots,0)$ and $c=\mathbf{r-1}$. So $g(X)= \binom{r}{2}$ as claimed.}

  \end{ex}

\begin{ex}
    \label{Example_Variation}
Suppose $n=3$ and $r=4$; this is the smallest case in which there are distinct possibilities for $M_X$. Let $v_i\in X_i\setminus\{0\}$; and suppose that $\sum_{i=1}^4a_iv_i=0$. There are two basic possibilities: either
\begin{enumerate}
   \item $a_i\neq0$ for every $i \in \{1,2,3,4\}$, in which case there are $6$ hyperplanes, given by $\langle e_i,e_j \rangle$ for $\{i,j\}\in\binom{[4]}{2}$; or else
    \item $a_i=0$ for some $i \in \{1,2,3,4\}$. In this case there are $4$ hyperplanes, given by $\langle e_i,e_j \rangle$ for $j \in [4]-\{i\}$ and $\langle e_i,e_j,e_k \rangle$.
\end{enumerate}
In the first case, $M_X$ is the uniform matroid $U_{[4]}^3$ of rank 3 on $[4]$. In the second case, we have $M_X=U_{[i]}^1\oplus U_{[j,k,l]}^2$, where $\{i,j,k,l\}=[4]$. This division into cases reflects the dichotomy of there existing a triple of coplanar lines or not.\footnote{There is one remaining loop-less matroid on $[4]$ of rank 3, namely $U_{[i,j]}^1\oplus U_{[k,l]}^2$. However, it is not geometrically realized by an essential line arrangement because it contains a pair of parallel elements $\{i,j\}$), and therefore is not simple.}

\emph{In the first case, we choose a parameterization for which the associated matrix $A$ satisfies}
\[
\mu_A(x_1)=(t_1,0,0,\al_4 t_4), \mu_A(x_2)=(0,t_2,0,\be_4 t_4), \text{ and }\mu_A(x_3)=(0,0,t_3,\ga_4 t_4)
\]
\emph{with $\al_4, \be_4, \ga_4 \in F^*$. Following the template of the previous example, we now characterize sets of homogeneous vectors of minimal support in every finite degree. By construction, these comprise all possible valuations of vector-valued polynomials}
\[
\begin{split}
P_d&= \sum_{j_1+j_2 \leq d} \ep_{j_1,j_2} (t_1,0,0,\al_4 t_4)^{j_1} (0,t_2,0,\be_4 t_4)^{j_2} (0,0,t_3,\ga_4 t_4)^{d-j_1-j_2} \\
&=(\ep_{d,0}t_1^d, \ep_{0,d}t_2^d, \ep_{0,0} t_3^d, \sum_{j_1+j_2 \leq d}\ep_{j_1,j_2} \al_4^{j_1} \be_4^{j_2} \ga_4^{d-j_1-j_2} t_4^d)
\end{split}
\]
\emph{where $\ep_{j_1,j_2} \in F$. It follows easily that at most two (resp. three) arbitrarily-chosen components of $P_1$ (resp. $P_2$) can be made to vanish.}

\emph{Consequently $c=\mathbf{2}=(2,2,2,2)$ is the conductor of ${\rm S}(X)$; while $(2,\dots,2,1,\dots,1)$ will arise as the valuative image of $P_d$ if and only if at least two of its components are equal to 1. It follows that $(1,0,0,0), (1,1,0,0), (1,1,1,0), (2,2,2,1)$ is the gap sequence of a saturated path linking $\mathbf{0}$ and $c=\mathbf{2}$. So $g(X)=4$.}

\emph{Similarly, in the second case, we choose a parameterization for which}
\[
\mu_A(x_1)= (t_1,0,0,0), \mu_A(x_2)= (0,t_2,0,\al_4 t_4) \text{ and } \mu_A(x_3)=(0,0,t_3,\be_4 t_4)
\]
\emph{with $\al_4,\be_4 \in F^*$. Correspondingly, we have}
\[
\begin{split}
P_d &= \sum_{j_1+j_2 \leq d} \ep_{j_1,j_2} (t_1,0,0,0)^{j_1} (0,t_2,0,\be_4 t_4)^{j_2} (0,0,t_3,\ga_4 t_4)^{d-j_1-j_2} \\
&=(\ep_{d,0}t_1^d, \ep_{0,d}t_2^d, \ep_{0,0} t_3^d, \sum_{j_2=0}^d\ep_{0,j_2} \be_4^{j_2} \ga_4^{d-j_2} t_4^d)
\end{split}
\]
\emph{where $\ep_{j_1,j_2} \in F$. The same analysis operative in the first case now applies; in particular, the semiring ${\rm S}(X)$ is unchanged, and once more $g(X)=4$.}

\end{ex}



\subsection{Value semirings of (very) uniform essential line arrangements}

In this subsection, we focus on the structure of value semirings associated with essential line arrangements whose underlying matroids are uniform. An auxiliary result will play a key role. To state it, we will make use of the following device.

\begin{dfn}
Fix positive integers $d$ and $n, r \geq 2$. Given scalars $\al_{i,j} \in F$ with $i \in \{1,\dots,r\}$ and $j \in \{1,\dots,n\}$, the associated \emph{coefficient matrix} $C_d(\al_{i,j})$ is the $r \times \binom{d+n-1}{n-1}$ matrix whose $i$-th line consists of all homogeneous degree-$d$ monomials in the $\al_{i,j}$, $j=1, \dots, n$ ordered lexicographically according to the convention that $\al_{i,j_1}>\al_{i,j_2}$ whenever $j_1>j_2$.
\end{dfn}

\begin{lemma}\label{interpolation_lemma}
Fix positive integers $d$ and $n, r \geq 2$. For a generic choice of scalars $\al_{i,j} \in F$ with $i \in \{1,\dots,r\}$ and $j \in \{1,\dots,n\}$, the associated coefficient matrix $C_d(\al_{i,j})$ is of maximal rank.
\end{lemma}
\begin{proof}
Let $q:=\min(r,\binom{d+n-1}{n-1})$. The coefficient matrix $C_d(\al_{i,j})$ is of maximal rank if and only if at least one of its $q \times q$ minors is nonvanishing.  Let $V=V(\al_{i,j})$ denote the determinantal variety cut out by all $q \times q$ minors. If $V$ is empty, there is nothing to prove; otherwise, $V$ defines a proper subvariety of the affine space with coordinates $\al_{i,j}$, and we conclude immediately.
\end{proof}

Lemma~\ref{interpolation_lemma} gives a partial geometric interpretation of the matroidal $\mb{B}$-semimodule ${\rm S}(X,d)$. Indeed, ${\rm S}(X,1)$ is the matroid underlying the line arrangement $X$, and is uniform if and only if the coefficient matrix $C_1(\al_{i,j})$ is of maximal rank.

\begin{dfn}
Given positive integers $n$ and $r$, let $d$ be the minimal positive integer for which $\binom{d+n-1}{n-1} \geq r$. An essential line arrangement $X$ with $r$ branches and embedding dimension $d$
is \emph{very uniform} if the associated coefficient matrices $C_e(\al_{i,j})$ are of maximal rank in every degree $1 \leq e \leq d$.
\end{dfn}

With Lemma~\ref{interpolation_lemma} in hand, determining the value semiring ${\rm S}(X)$ of a very uniform essential line arrangement with $r$ branches and embedding dimension $n$ in any particular case is straightforward. 
Namely, Lemma~\ref{interpolation_lemma} implies that whenever $\binom{d+n-1}{n-1} \geq r$, the vector-valued polynomial $P_d$ introduced in the previous subsection assumes as values arbitrary permutations of $(d,0,\dots,0)$. 
On the other hand, when $\binom{d+n-1}{n-1} < r$, $P_d$ assumes arbitrary permutations of $(d,\dots, d, 0, \dots, 0)$, in which the number of instances of $d$ is at least $r-1-\binom{n+d-1}{n-1}$. The following consequences for ${\rm S}(X)$ are immediate.

\begin{thm}\label{value_semiring_very_uniform_case}
Given positive integers $n$ and $r$, let $X$ denote a very uniform essential line arrangement with $r$ branches and embedding dimension $n$. Let $d$ be the minimal positive integer for which $\binom{d+n-1}{n-1} \geq r$; the conductor ${\rm S}(X)$ is then $c=\mathbf{d}=(d,\dots,d)$. On the other hand, for every positive integer $d$ for which $\binom{d+n-1}{n-1} < r$, arbitrary permutations of the vector $(d,\dots,d,d-1,\dots,d-1)$ featuring $i_d \in \{d,\dots,r-1\}$ instances of $d$, is a gap of ${\rm S}(X)$; and conversely, every gap is of this form.
\end{thm}

While Theorem~\ref{value_semiring_very_uniform_case} gives a complete characterization of ${\rm S}(X)$, the combinatorial structure of ${\rm S}(X)$ requires further elucidation whenever $n>2$. One would hope for explicit bounds on the genus of ${\rm S}(X)$.

\subsection{Conditions on rational curves imposed by $r$-tuple points}

In this subsection, we apply the results of the preceding subsections towards the dimension theory of {\it Severi varieties} of rational curves with singleton singular loci that are multiple points. We abusively conflate rational curves with their parameterizing morphisms.

\begin{dfn}
Given a reduced and irreducible algebraic curve $C \sub \mb{P}^n$, we say that a point $p \in C$ is an \emph{$r$-tuple point of embedding dimension $\ell$} if the tangent lines to the branches of $C$ in $p$ determine an essential line arrangement of embedding dimension $\ell$.
\end{dfn}

Given positive integers $d$ and $n$, let $M^n_d$ denote the space of morphisms $f: \mb{P}^1 \ra \mb{P}^n$ of degree $d$.

\begin{lemma}\label{codimension_lemma}
Assume that $d>r$. The Severi subvariety of curves with singleton singular loci that are $r$-tuple points of embedding dimension $\ell$ has codimension at most $(n-\ell)(r-\ell)+(r-1)n-r$ in $M^n_d$.
\end{lemma}

\begin{proof}
Fix a choice of $r$ preimages $p_1,\dots,p_r$ and target $q$ of an $r$-tuple point in the image of a morphism $f=(f_0,\dots,f_n): \mb{P}^1 \ra \mb{P}^n$; doing so imposes $rn$ independent linear conditions on the coefficients of the parameterizing functions $f_i$. Explicitly, assume without loss of generality that $q=(1:0:\cdots:0)$; that $p_j$ corresponds to $t=\al_j \in F \cup \{\infty\}$ with respect to an affine coordinate $t$ on $\mb{P}^1$; and that $f_i(t)=\sum_{k=0}^d a_{i,k} t^k$ for every $i=1,\dots,n$. Whenever $\al_j \neq \infty$ (resp., $\al_j=\infty$), the $r$-tuple point imposes the vanishing condition $\sum_{k=0}^d a_{i,k} \al_j^k=0$ (resp., $a_{i,d}=0$) on the coefficients of each parameterizing function $f_i$, $i=1,\dots,n$. The fact that these conditions are independent as $j=1,\dots,r$ varies follows from the fact that the associated coefficient matrix whose $j$-th line is $(\al_j^d,\dots,\al_j,1)$ (resp., $(0,\dots,0,1)$) is essentially a (submatrix of a) Vandermonde matrix, and therefore of maximal rank, whenever $d>r$.

\medskip
Similarly, whenever $\al_j \neq \infty$ (resp., $\al_j=\infty$), the tangent vector to the $j$-th branch of $f$ in $q$ in the affine chart where the zeroth homogeneous coordinate of the ambient $\mb{P}^n$ is nonvanishing is $(\sum_{k=0}^{d-1} (d-k) a_{1,k}\al_j^{d-1-k},\dots,\sum_{k=0}^{d-1} (d-k) a_{n,k}\al_j^{d-1-k})$ (resp., $(\sum_{k=1}^d k a_{1,k}, \dots, \sum_{k=1}^d k a_{n,k})$). That the corresponding line arrangement should have embedding dimension $\ell$ amounts to requiring that the $r \times n$ matrix whose $j$-th row is the tangent vector to the $j$-th branch of $f$ in $q$ have rank $\ell$. For a {\it generic} $r \times n$ matrix, the associated codimension is $(n-\ell)(r-\ell)$; see, e.g., \cite[Ch. II, Sec. 2]{ACGH}. So this is an upper bound on the number of algebraically independent conditions imposed by the the requirement that an $r$-tuple multiple point with fixed preimages and target have embedding dimension $\ell$.

\medskip
Finally, allowing preimages and target of our $r$-tuple point to vary, we lose $n+r$ degrees of freedom, and the required estimate follows immediately.
\end{proof}


As a corollary, we obtain infinitely many interesting (and new) examples of Severi varieties of unisingular rational curves of unexpectedly large dimension. 

\begin{coro}\label{excess_severi_examples}
Assume that $d>r$, and that $r \gg 0$. The Severi subvariety of curves with singleton singular loci that are $r$-tuple points of embedding dimension $2$ has codimension strictly less than $(n-2)g$ in $M^n_d$, where $g$ is the delta-invariant of the corresponding planar $r$-tuple multiple point; in particular, this Severi variety is reducible.
\end{coro}

\begin{proof}
Indeed, suppose $\ell=2$. According to Lemma~\ref{codimension_lemma}, the Severi variety of unisingular rational curves with planar $r$-tuple points is of codimension at most $(r-2)(n-2)+(r-1)n-r$, and will have excess dimension whenever the latter quantity is strictly less than $(n-2)g$, the codimension of the (closure of the) locus of $g$-nodal rational curves; see, e.g., the introduction of \cite{CLV}. But according to Example~\ref{planar_case} we have $g=\binom{r}{2}$; while for every fixed positive integer $n \geq 3$ we have
\[
(r-2)(n-2)+(r-1)n-r < (n-2) \binom{r}{2}
\]
for $r \gg 0$. Reducibility of the Severi variety follows immediately.
\end{proof}

\section*{Acknowledgements}
Parts of this work were completed while during visits of C.G. to IMECC-Unicamp in 2022, and of E.C. to Cinvestav-IPN in 2023. We thank both of these institutions for their hospitality and financial support. We also thank Marcelo Escudeiro Hernandes for illuminating conversations and an invitation to speak on related work.

\end{document}